\documentclass{amsart}
\usepackage[utf8]{inputenc}
\usepackage{amsmath,eucal,graphicx,amsxtra,hyperref}
\usepackage{amsrefs}

\input xy
\xyoption{all}

\newtheorem{theorem}{Theorem}[section]
\newtheorem{lemma}[theorem]{Lemma}

\theoremstyle{definition}
\newtheorem{definition}[theorem]{Definition}
\newtheorem{example}[theorem]{Example}

\newtheorem{proposition}[theorem]{Proposition}
\newtheorem{corollary}[theorem]{Corollary}

\theoremstyle{remark}
\newtheorem{remark}[theorem]{Remark}

\newcommand{\join}{\#}
\newcommand{\tens}{\tilde{\otimes}}

\newcommand{\Cyc}{{\mathcal C}}
\newcommand{\ZZ}{{\mathcal Z}}
\newcommand{\rtens}{\hat{\otimes}}

\newcommand{\C}{\mathbb{C}}

\newcommand{\Z}{\mathbb{Z}}

\newcommand{\Q}{\mathbb{Q}}

\renewcommand{\S}{\mathbb{S}}
\renewcommand{\P}{\mathbb{P}}
\newcommand{\PP}{\mathbb{P}}
\newcommand{\GG}{\mathbb{G}}

\def\S{{\mathbb S}}
\def\G{{\mathcal G}}
\def\GG{{\mathcal G}}

\def\BU{{\bf BU}}

\def\Cic{{\mathcal C}}

\begin{document}
\title[Non-existence of Tensor Products]{On the non-existence of Tensor Products of Algebraic Cycles}
\date{\today}
\author[L. \,E. Lopez]{Luis E. Lopez}
\address{Max-Planck-Institut f\"ur Mathematik\\
         Vivatsgasse 7\\
	 D-53111 Bonn \\
	 Germany
}
\curraddr{Max-Planck-Institut f\"ur Mathematik\\
         Vivatsgasse 7\\
	 D-53111 Bonn \\
	 Germany
}
\email{llopez@mpim-bonn.mpg.de}
\thanks{This work is based on part of the author's dissertation. It is a pleasure to thank H. Blaine Lawson for all
his guidance and helpful remarks.}
\begin{abstract} Let $\otimes$ be the map which classifies the tensor product of two line bundles, an extension
of this map to the space of all codimension $1$ algebraic cycles is constructed. It is proved that this extension
cannot exist in codimension greater than $1$.
\end{abstract}
\maketitle


\section{Introduction}

Boyer, Lawson, Lima-Filho, Mann and Michelsohn 
settled the Segal conjecture
in \cite{BLLMM}. One of the fundamental results which motivated the proof
is that there is a geometric construction which extends the map classifying
the direct sum of vector bundles in $\BU $ to the space $\ZZ$ of all algebraic cycles, 
namely, the linear join $\join$ of cycles, i.e. the following diagram commutes:
$$\xymatrix{
\BU \times \BU \ar[r]^-{\oplus} \ar[d]_c & \BU \ar[d]^c \\
\ZZ \times \ZZ \ar[r]^-{\join} & \ZZ 
}
$$
where $c$ is the total chern class map. \\

In \cite{BLLMM} the authors mention that one would like to have a geometric
construction on the space of algebraic cycles which extends the tensor
product in the level of $\BU$ (i.e. degree one cycles). Segal 
proved in \cite{Segaltensor} that $\BU$ has an infinite loop
space structure where the $H$-space structure is induced by the map
classifying the tensor product of vector bundles. 
Therefore the construction requested by the authors of \cite{BLLMM} would
give yet another infinite loop space structure on $\ZZ$. 
We will provide an idea of the extent to which such construction is possible:

\begin{theorem} \label{pairing_divisors} There is an algebraic biadditive pairing 
$\rtens$ which extends the tensor product  to all effective divisors:
$$\xymatrix{
	\Cyc^1_d(\PP^n) \times \Cyc^1_e(\PP^m) \ar@{->}[r]^-{\rtens} & \Cyc^{1}_{de}(\PP^{mn+m+n}) 
	}$$
\end{theorem}
This product is constructed via an algebraic pairing in the corresponding
rings of polynomials which may be of interest in its own right. The formula
obtained from stabilizing and group-completing the pairing to the stabilized space $\ZZ_{0}^1(\PP^{\infty})$
of algebraic cycles of codimension $1$ and degree $0$  
yields a commutative diagram  
$$\xymatrix{
  \GG^1(\PP^{\infty}) \times \GG^1(\PP^{\infty}) \ar[r] \ar[d]_{c_1} & \GG^1(\PP^{\infty}) \ar[d]^{c_1} \\
	\ZZ^1_0(\PP^{\infty}) \times \ZZ^1_0(\PP^{\infty}) \ar@{->}[r]^-{\rtens} & \ZZ^{1}_0(\PP^{\infty}) 
	}$$
which recovers the group 
structure in the second cohomology group given by the tensor product of line bundles
$$ c_1(L_1\otimes L_2) = c_1(L_1) + c_1(L_2) \in H^2(BU_1)$$ 

The Hurewicz map is the main tool in proving that a general pairing does
not exist. The following theorem calculates the classes pulled back by  the Hurewicz map.

\begin{theorem} \label{hurewicz} The inclusion of the grassmannian $\GG^1(\PP^n)$ of 
hyperplanes  in $\PP^n$ into the space $\ZZ^1(\PP^n)$ of all cycles in $\PP^n$ factors
through the free group $\Z \GG^1(\PP^n)$:
\begin{equation}
\label{factor}
\xymatrix{
\GG^1(\PP^n) \ar@{^{(}->}[r]^i & 
\Z\GG^1(\PP^n)\ar@{^{(}->}[r]^j \ar[d]^= & \ZZ^1(\PP^n) \ar[d]^= \\
 & \prod_{j=0}^n K(\Z,2j) & K(\Z,0)\times K(\Z,2)
}
\end{equation}

With respect to the canonical product decomposition given in (\href{factor}), the map
$i$ classifies the cohomology class $1\times \omega \times \cdots \times \omega^n$ where
$\omega$ is the multiplicative generator of $H^2(\GG^1(\PP^n))$  and
the map $j$ is homotopic to the projection $\pi_0 \times \pi_1$

\end{theorem}
The pairing constructed for divisors in (\ref{pairing_divisors}) cannot be extended to a
continuous biadditive pairing on the space of cycles of higher codimension, but it does admit
an extension if we restrict the second factor of the pairing to the subgroup $\Z\GG^1(\PP^m)$ 
of cycles which are unions of hyperplanes (possibly with multiplicities). 

\begin{theorem}There is a continuous biadditive pairing $\tens$ which makes the following diagram commute
$$\xymatrix{
\GG^p(\PP^n) \times \GG^1(\PP^m) \ar[r]^{\otimes} \ar[d]_{c \times h}
& \GG^p(\PP^{nm +n +m}) \ar[d]^c \\
\ZZ^p_0(\PP^n) \times \Z_0\GG^1(\PP^m) \ar[r]^-{\tens}
& \ZZ^p_0(\PP^{nm+n+m})
}$$
\end{theorem}

The relevance of this diagram is twofold, on the one hand it provides a new way of calculating
the formula for the total chern class of the tensor product of a vector bundle and a line
bundle, namely
\begin{equation} \label{one}
 c_i(E\otimes L) = \sum_{j=0}^i {{rk(E) - j} \choose {i-j}}c_j(E)c_1(L)^{i-j}
 \end{equation} 
and on the other hand it suggests a path for  generalizing the Bott periodicity map which is related to 
the top arrow of this diagram (the problem for generalizing the Bott map is that there 
is no "orthogonal complement" in the space of cycles.) \\
The formula \ref{one} does not describe completely the map induced in cohomology by the pairing $\tens$, since
$h^{*}(i_{2k}^m)=h^{*}(i_{2t}^s)=\omega^{km}$ if $km=ts$. This calculation can actually be obtained
rationally:
\begin{theorem} \label{rationalcalculation}
The pairing 
$$
\ZZ^p_0(\PP^n) \times \Z_0\GG^1(\PP^m) \rightarrow \ZZ^p(\PP^{nm+n+m})
$$
induces the following map in rational cohomology
$$
\tens^*(i_{2k})= \sum_{j=0}^k{{p-j} \choose{k-j}}i_{2j} \otimes \tilde{i}_{2(i-j)}
$$
where $i_{2k}$ is the fundamental class of the $k$-th factor of $\ZZ^p_0(\PP^n) \simeq \prod_{j=1}^p K(\Z,2j)$
and $\tilde{i}_{2l}$ is the fundamental class in the $l$-th factor of
$\Z_0\GG^1(\PP^m) \simeq \prod_{j=1}^m K(\Z,2j)$
\end{theorem} 

Using  theorems \ref{hurewicz} and \ref{rationalcalculation} the following 
theorem is proved

\begin{theorem}
\label{pairing}
 There is no continuous biaddditive pairing in the stabilized space of cycles
which makes the following diagram commute
$$\xymatrix{
\BU \times \BU \ar[r]^-{\otimes} \ar[d]_{c \times c}
& \BU \ar[d]^c \\
\ZZ \times \ZZ \ar[r]^-{\tens}
& \ZZ
}$$

\end{theorem}

\begin{remark}
Burt Totaro used the projection formula  in \cite{tot} to prove that the total chern class 
map does not extend to a map of cohomology theories. This is also implied by the results in
this paper.
\end{remark}

%
                         
\section{Tensor Pairing for divisors}

In this section we will define a product
$$\tens : \Cic^1(\P^{n-1}) \times \Cic^1(\P^{m-1})\rightarrow
\Cic^1(\P^{mn-1}) $$
which is continuous and biadditive (each of the factors has the structure
of a topological monoid). This pairing generalizes the pairing on
cycles of degree $1$ which classifies the tensor product of the universal
quotient bundle. \\

 Let $\C[\bar{x}]_d := \C[x_0,\ldots,x_{n-1}]_d$ denote the set
of complex polynomials of degree $d$ in the variables $x_0,\ldots,  x_{n-1}$.
This set is a complex vector space of dimension $N= \binom{n+d}{n}$. If we order the variables
lexicographically, we get the following ordered basis for this vector space:
$${\mathcal B}:=\{x_0x_0\cdots x_0, \ldots, x_{j_1}x_{j_2}\cdots x_{j_{n-1}},
\ldots,x_{n-1}\cdots x_{n-1} \}$$
where $j_1 \leq j_2 \leq \ldots \leq j_{n-1}$, i.e. the set of all monomials
of degree $d$ ordered lexicographically. 

\begin{definition} Let 
$$\Psi_{de}: \underbrace{\C[\bar{x}]_d \times \cdots \times  
\C[\bar{x}]_d}_{m- \mathrm{times}}
 \times  \underbrace{ \C[\bar{y}]_e \times  \cdots \times \C[\bar{y}]_e}_
{n-\mathrm{times}} 
\rightarrow \C[\bar{z}]_{de}$$
be the multilinear homomorphism defined on the elements of the bases 
$\mathcal B$ by
\begin{multline*}\Psi_{de}(x_{j_1^1}\cdots x_{j_d^1},\ldots, x_{j_1^e}\cdots
x_{j_d^e},y_{k_1^1}\cdots y_{k_e^1},\ldots,y_{k_1^d} \cdots
y_{k_e^d}) = \\
 z_{j_1^1k_1^1}\cdots z_{j_d^1k_1^d} \cdots z_{j_1^ek_e^1}\cdots 
z_{j_d^ek_e^d}
\end{multline*} 
and extended multilinearly.
\end{definition}

Notice that this definition depends on the ordered bases  $\mathcal B$, in particular, a 
different order on the elements would yield a different homomorphism.

The function $\Psi_{1e}$  has a particularly nice expression when the
degree $1$ forms are monomials. 

\begin{lemma} \label{psilinearmonomial}
$$\Psi_{1e}(x_{i_1},\ldots,x_{i_e},g) =g(z_{i_1j^1},\ldots,z_{i_ej^e})$$
\end{lemma}
\begin{proof} Let $g = \sum b_{j_1 \cdots j_e} y_{j_1}\cdots y_{j_e}$. Then, 
following the definition we get
\begin{multline*}
\Psi(x_{i_1},\ldots,x_{i_e},g) =
\Psi(x_{i_1},\ldots,x_{i_e},\sum b_{j_1 \cdots j_e} y_{j_1}\cdots y_{j_e}) =\\ 
\sum b_{j_1 \cdots j_e} \Psi(x_{i_1},\ldots,x_{i_e},y_{j_1}\cdots y_{j_e}) =
\sum b_{j_1 \cdots j_e} z_{i_1j_1}\cdots z_{i_ej_e}=\\
g(z_{i_1j_1},\ldots,z_{i_ej_e})
\end{multline*}

\end{proof}

With the previous definition, we can now define the tensor product of divisors:

\begin{definition} Given  $f \in \C[x_0,\ldots,x_{n-1}]_d$ and $g \in \C[y_0,\ldots,y_{m-1}]_e$
we define $f \tens g \in \C[\ldots, z_{jk},\ldots]$ by
$$f \tens g := \Psi_{de}
( \underbrace{f,\ldots,f}_{e-\mathrm{times}}, \underbrace{g,\ldots, g}_{d-\mathrm{times}})
$$
\end{definition}

\begin{example} Let $f= x_0^2-3x_1x_2$ and $g= y_5y_7$. Then 
\begin{multline*}
f \tens g = \Psi(f,f,g,g)=\Psi(x_0^2-3x_1x_2,x_0^2-3x_1x_2,y_5y_7,y_5y_7)=\\
\Psi(x_0^2,x_0^2-3x_1x_2,y_5y_7,y_5y_7) - 3\Psi(x_1x_2,x_0^2-3x_1x_2,y_5y_7,y_5y_7)=\\
\Psi(x_0^2,x_0^2,y_5y_7,y_5y_7) - 3\Psi(x_0^2,x_1x_2,y_5y_7,y_5y_7)\\
-3\Psi(x_1x_2,x_0^2,y_5y_7,y_5y_7)+9\Psi(x_1x_2,x_1x_2,y_5y_7,y_5y_7)=\\
z_{05}z_{05}z_{07}z_{07}-3z_{05}z_{05}z_{17}z_{27}
-3z_{15}z_{25}z_{07}z_{07}+9z_{15}z_{25}z_{17}z_{27}
\end{multline*}
\end{example}

The next result is the first step towards proving that the pairing is indeed
 biadditive.

\begin{proposition} 

\begin{multline*}
\Psi_{rm}(f_1,\ldots,f_m,g,\ldots,g)
\Psi_{sm}(\phi_1,\ldots,\phi_m,g,\ldots,g)=\\
\Psi_{(r+s)m}(f_1\phi_1,\ldots,f_m\phi_m,g,\ldots,g)
\end{multline*}

\end{proposition}
\begin{proof}
Suppose that

\begin{multline*}
\Psi_{rm}(f_1,\ldots,f_m,g,\ldots,g)\Psi_{sm}(\phi_1,\ldots,\phi_m,g,\ldots,g)=\\
\Psi_{(r+s)m}(f_1\phi_1,\ldots,f_m\phi_m,g,\ldots,g)
\end{multline*}

and

\begin{multline*}
\Psi_{rm}(F,f_2,\ldots,f_m,g,\ldots,g)\Psi_{sm}(\phi_1,\ldots,\phi_m,g,\ldots,g)=\\
\Psi_{(r+s)m}(F\phi_1,f_2\phi_2, \ldots,f_m\phi_m,g,\ldots,g)
\end{multline*}

then it follows from the multilinearity of $\Psi_{ij}$ that
\begin{multline*}
\Psi_{rm}(f_1+cF,f_2,\ldots,f_m,g,\ldots,g)\Psi_{sm}(\phi_1,\ldots,\phi_m,g,\ldots,g) =\\
\left[\Psi_{rm}(f_1,\ldots,f_m,g,\ldots,g)+\Psi_{rm}(cF,f_2,\ldots,f_m,g,\ldots,g) \right]
\Psi_{sm}(\phi_1,\ldots,\phi_m,g,\ldots,g)= \\ \Psi_{(r+s)m}(f_1\phi_1,f_2\phi_2,\ldots,f_m\phi_m,g,\ldots,g)+
\Psi_{(r+s)m}(f_1\phi_1+cF\phi_1,f_2\phi_2,\ldots,f_m\phi_m,g\ldots,g)=\\
\Psi_{(r+s)m}((f_1+cF)\phi_1,\ldots,f_m\phi_m,g,\ldots,g)
\end{multline*}
therefore, it suffices to prove the statement in the case that
$f_1$ is a monic monomial. Analogously, it suffices to prove the statement
in the case that every $f_i$ is a monomial and $\phi_i$ is a monomial. That is,
we must show 
\begin{multline} \label{reduction}
\Psi_{rm}(x_{i_1^1}\cdots x_{i_r^1},\ldots,x_{i_1^m}\cdots x_{i_r^m},g,\ldots,g) 
\Psi_{sm}(x_{k_1^1}\cdots x_{k_s^1},\ldots,x_{x_1^m}\cdots x_{i_s^m},g,\ldots,g)=\\
\Psi_{(r+s)m}(x_{i_1^1}\cdots x_{i_r^1}x_{k_1^1}\cdots x_{k_s^1},\ldots, x_{i_1^m} \cdots
x_{i_r^m}x_{k_1^m}\cdots x_{k_m^s},g,\ldots,g)
\end{multline}
without loss of generality we may assume that $x_{i_a^w} \leq x_{i_b^w}$ if $a \leq b$
and $x_{k_a^w} \leq x_{k_b^w}$ if $a \leq b$. We will prove equation
\ref{reduction} by induction on $r$ and $s$. \\

{\bf Base: $r=1$ and $s=1$.} Notice that by lemma \ref{psilinearmonomial}

$$\Psi(x_{i^1_1},\ldots, x_{i^m_1},g) = g(z_{i_1^1j_1^1},
\ldots, z_{i^m_1j^1_m})$$
therefore,
\begin{multline*}
\Psi(x_{i_1^1},\ldots,x_{i_1^m},g)\Psi(x_{k^1_1},\ldots,x_{k_1^m},g) =\\
g(z_{i_1^1j_1^1},\ldots, z_{i^m_1j^1_m}) g(z_{i_1^1k_1^1}, \ldots, z_{i^m_1k^1_m})
\end{multline*}
On the other hand, if
$$g(y_1,\ldots,y_m)= \sum a_{j_1 \cdots j_r}y_{j_1}\cdots y_{j_r}$$
then
\begin{multline} \label{eq33}
\Psi(x_{i^1_1}x_{k^1_1},\ldots,x_{i^m_1}x_{k^m_1},g,g)=\\
\Psi(x_{i^1_1}x_{k^1_1},\ldots,x_{i^m_1}x_{k^m_1},
\sum a_{j_1^1 \cdots j_r^1}y_{j_1^1}\cdots y_{j_r^1},
\sum a_{j_1^2 \cdots j_r^2}y_{j_1^2}\cdots y_{j_r^2})=\\
\sum a_{j_1^1 \cdots j_r^1}a_{j_1^2 \cdots j_r^2}
\Psi(x_{i^1_1}x_{k^1_1},\ldots,x_{i^m_1}x_{k^m_1},
y_{j_1^1}\cdots y_{j_r^1},
y_{j_1^2}\cdots y_{j_r^2})= \\
\sum a_{j_1^1 \cdots j_r^1}a_{j_1^2 \cdots j_r^2}
z_{\sigma_1^1j_1^1} \cdots z_{\sigma_1^mj_m^1}
z_{\tau_1^1j^2_1}\cdots z_{\tau_1^mj_m^2} 
\end{multline}
where 
\begin{equation}
\nonumber
\sigma_s^t =
\begin{cases}
i_s^t, & \text{if $i_s^t \leq k_s^t$} \\
k_s^t, & \text{if $i_s^t > k_s^t$}
\end{cases}
\; \; \text{and} \; \; 
\tau_s^t =
\begin{cases}
k_s^t, & \text{if $i_s^t \leq k_s^t$} \\
i_s^t, & \text{if $i_s^t > k_s^t$}
\end{cases}
\end{equation}
Now, notice that if we exchange the definition of $\sigma$ and $\tau$ the sum
on the right hand side of \ref{eq33} remains unchanged. This happens because
we are taking two copies of $g$. Therefore we may assume without loss of 
generality that $\sigma_s^t = i_s^t$ and $\tau_s^t = k_s^t$. In this case,
equation \ref{eq33} becomes
\begin{multline} \label{eq34}
\Psi(x_{i^1_1}x_{k^1_1},\ldots,x_{i^m_1}x_{k^m_1},g,g)=\\
\sum a_{j_1^1 \cdots j_r^1}a_{j_1^2 \cdots j_r^2}
z_{i_1^1j_1^1} \cdots z_{i_1^mj_m^1}
z_{k_1^1j^2_1} \cdots z_{k_1^mj_m^2} = \\
\left(\sum a_{j_1^1 \cdots j_r^1}
z_{i_1^1j_1^1} \cdots z_{i_1^mj_m^1}
\right)
\left(\sum a_{j_1^2 \cdots j_r^2}
z_{k_1^1j^2_1} \cdots z_{k_1^mj_m^2}
\right)= \\
g(z_{i_1^1j_1^1},\ldots, z_{i^m_1j^1_m}) g(z_{i_1^1k_1^1}, \ldots, z_{i^m_1k^1_m})
\end{multline}
and the base of the induction is proved. \\
{\bf Inductive step:} Essentialy the same argument proves that both sides of
\ref{reduction} are equal to

\begin{multline} \label{reduction2}
g(z_{i_1^1j_1^1},\ldots,z_{i_1^mj_m^1})\cdots 
g(z_{i_r^1j_1^r},\ldots,z_{i_r^mj_m^r})\cdot \\
g(z_{k_1^1j_1^1},\ldots,z_{k_1^mj_m^1})\cdots
g(z_{k_s^1j_1^s},\ldots,z_{k_s^mj_m^s})
\end{multline}
\end{proof}

\begin{corollary}
$$(f_1f_2)\tens g = (f_1 \tens g)(f_2 \tens g)$$
\end{corollary}
\begin{proof}
This follows immediately from the definitions and the proposition.
\end{proof}
Analogously the following is true
\begin{corollary}
$$f\tens (g_1 g_2) = (f \tens g_1)(f \tens g_2)$$
\end{corollary}

This theorem gives some insight into the geometry of the hypersurfaces
obtained as $\tens$ products. The next lemma provides the description
in the case that one of the factors is a linear space.

\begin{lemma} \label{tensorlinearform}
If $f=\sum a_i x_i$ then 
$$f \tens g = g\left(f\left(z_{00},\ldots,z_{(n-1)0}\right),\ldots,
f\left(z_{0(m-1)},\ldots,z_{(n-1)(m-1)}\right)\right)$$
That is, the codimension $1$ cycle defined by $f \tens g$ is
isomorphic to the linear join of a linear space and the cycle
defined by $g$ (Using Lawson's terminology, it is the iterated
suspension of the cycle defined by $g$)
\end{lemma}
\begin{proof} Let $g= \sum a_{j_1\cdots j_e} y_{j_1}\cdots y_{j_e}$. Then,
\begin{multline*}
f \tens g = \Psi(f,\ldots,f,g)=
\Psi(f,\ldots,f,\sum a_{j_1\cdots j_e} y_{j_1}\cdots y_{j_e}) =\\
\sum a_{j_1\cdots j_e} \Psi(f,\ldots,f, y_{j_1}\cdots y_{j_e})=
\sum a_{j_1\cdots j_e} f \tens (y_{j_1}\cdots y_{j_e}) =\\
\sum a_{j_1\cdots j_e} (f \tens y_{j_1})\cdots (f \tens y_{j_e})
\end{multline*}
This last expression is exactly what we are looking for, it says that
we should substitute in the polynomial $g$ the variable $y_{j_i}$ with
the polynomial $f \tens y_{j_i}$ which in turn is equal to the
polynomial $f$ evaluated in the variables $z_{0j_i},\ldots,z_{(n-1)j_i}$
\end{proof}

We recall that there is a one to one correspondence between  homogeneous
polynomials in the variables $x_0,\ldots, x_s$ and the 
codimension $1$ algebraic cycles in $\P^s$. The correspondence is given in
the following way: If $f$ is a polynomial and it decomposes as a product
$f_1^{\alpha_1}\cdots f_t^{\alpha_t}$ where each $f_k$ is irreducible,
then the corresponding cycle $\Cyc(f)$ is given by $\sum \alpha_i V(f_i)$
where $V(f_i)$ is the (necessarilly irreducible) variety defined by the
polynomial $f_i$. The disjoint union of all codimension cycles $\Cyc^1(\P^s)$ forms a
monoid with respect to the formal addition of cycles. The following theorem
expresses the results of this section in terms of cycles.

\begin{theorem} There is an algebraic pairing $\tens$ in the space of codimension
$1$ cycles in projective space:
$$\tens : \Cic^1(\P^{n-1}) \times \Cic^1(\P^{m-1})\rightarrow
\Cic^1(\P^{mn-1}) $$
which satisfies the following properties
\begin{enumerate}
\item $\tens$ coincides with the tensor product $\otimes$ on linear cycles.
\item $\tens$ is biadditive: 
$$\eta_1 \eta_2 \tens \xi = \eta_1 \tens \xi + \eta_2 \tens \xi$$ and
$$\eta \tens \xi_1 \xi_2 = \eta \tens \xi_1 + \eta \tens \xi_2$$
\item $\tens$ stabilizes to a pairing 
$$\tens : \Cyc^1(\P^{\infty}) \times \Cyc^1(\P^{\infty})\rightarrow
\Cyc^1(\P^{\infty}) $$
\end{enumerate}
\end{theorem}

Since the pairing is biadditive, it induces a pairing in the group
completion 
$$\tens : \ZZ^1(\P^{\infty}) \times \ZZ^1(\P^{\infty})\rightarrow
\ZZ^1(\P^{\infty}) $$

In order to get a commutative diagram we construct
the associated reduced pairing $\rtens$:
$$\rtens(\eta,\xi):= \eta \tens \xi + \eta \tens \xi_0 + \eta_0 \tens \xi$$
where $\xi_0$ and $\eta_0$ are two fixed hyperplanes.

\begin{theorem} \label{diagramlinetensor} The following diagram commutes
$$\xymatrix{
  \GG^1(\PP^{\infty}) \times \GG^1(\PP^{\infty}) \ar[r] \ar[d]_{i} & \GG^1(\PP^{\infty}) \ar[d]^{i} \\
	\ZZ^1(\PP^{\infty}) \times \ZZ^1(\PP^{\infty}) \ar@{->}[r]^-{\rtens} & \ZZ^{1}
	(\PP^{\infty}) }$$
\end{theorem}
It is proved in \cite{LM} that 
the space $\ZZ^1(\P^{\infty})$ splits as $\Z \times \ZZ_0^1(\P^{\infty})$,
where $\ZZ_0^1(\P^{\infty})$ is the subgroup of all cycles of degree zero.\\

Since we know that $\deg(\eta \tens \xi)=\deg(\eta)\deg(\xi)$ we only have
to calculate what happens with the pairing $\rtens$ when we restrict it to
cycles of degree $0$. Lawson proved that $\ZZ^1_0(\P^{\infty})$ is an
Eilenberg-Maclane space of type $K(\Z,2)$. Using this fact we will show
that the pairing $\tens$ restricted to the subgroup of cycles of
degree zero is nullhomotopic.

\begin{theorem} Any continuous biadditive pairing 
$$\tens: \ZZ^1_0(\P^{\infty}) \times \ZZ_0^1(\P^{\infty})
\rightarrow \ZZ^1_0(\P^{\infty})$$
is nullhomotopic.
\end{theorem}
\begin{proof} Since the pairing is biadditive it factors through
the smash product 
$$\ZZ^1_0(\P^{\infty}) \wedge \ZZ^1_0(\P^{\infty}) \rightarrow
\ZZ_0^1(\P^{\infty})$$
homotopically this last function is equivalent to a map
$$K(\Z,2) \wedge K(\Z,2) \rightarrow
K(\Z,2)$$
Now, notice that $K(\Z,2) \wedge K(\Z,2)$ is a CW-complex with cells only
in dimension $4$ and higher, therefore, the pullback in the second 
cohomology groups of the 
fundamental class in $K(\Z,2)$ is zero.
\end{proof}

This theorem allows us to calculate the class pulled back via the
pairing $\rtens$.

\begin{corollary} Let $\rtens$ be the pairing
$$\rtens(\eta,\xi):= \eta \tens \xi + \eta \tens \xi_0 + \eta_0 \tens \xi$$
and let $i_2$ be the fundamental class in $H^2(\ZZ^1(\P^{\infty});\Z)$. Then
$$\rtens^*(i_2) = i_2 \otimes i_0 + i_0 \otimes i_2$$
\end{corollary}
\begin{proof} The formula follows at once from pulling back the last two summands
of the  pairing $\rtens$, since by the previous theorem the first summand is
nullhomotopic.
\end{proof}

\begin{corollary} Let $L_1$ and $L_2$ be two line bundles. Then
$$c_1(L_1 \otimes L_2) = c_1(L_1) + c_1(L_2)$$
where $c_1$ denotes the first chern class.
\end{corollary}
\begin{proof}
Lawson and Michelsohn proved in \cite{LM} that the inclusion $i$
in theorem \ref{diagramlinetensor} classifies the chern class of the
universal quotient bundle. Therefore the formula follows from the previous
corollary.
\end{proof}

                            %
\section{Topological Obstruction for a General Pairing}

The pairing constructed in the last
section might be considered as a hint for a pairing in higher 
codimensions. We will prove that the  is a topological obstruction
for the existence of such a pairing. The general strategy is to
factor the inclusion of the grassmannian 
$\GG^p(\P^n)$ into the space $\ZZ^p(\P^n)$ of all codimension $p$
cycles. This inclusion factors through the free abelian  group 
$\Z \GG^p(\P^n)$ generated by the points of the grassmannian. The
existence of this factorization and the chern class formula for
the tensor product of bundles will yield a contradiction if we
assume the existence of a pairing. \\

Let us start by observing that the Dold-Thom theorem implies that
the free abelian group $\Z \GG^1(\P^n)$ is homotopically equivalent
to the product $\prod_{i=0}^n K(\Z,2i)$. Also, if we consider the
subgroup $\Z_0 \GG^1(\P^n)$ which is the kernel of the degree 
homomorphism i.e. the subgroup of $0$-dimensional cycles of 
degree $0$, we get the following homotopy equivalence
$$
\Z_0 \GG^1(\P^n) \simeq \prod_{i=1}^n K(\Z,2i)
$$        

Dold and Thom also proved that the inclusion 
$$i: \GG^1(\P^n) \hookrightarrow \Z_0 \GG^1(\P^n)$$
induces the Hurewicz map when the $\pi_i$ functors are applied. The
next theorem calculates the class pulled back in cohomology by the
inclusion $i$.

\begin{theorem}\label{huerewicz} The Hurewicz map
$$h: \GG^1(\P^n) \hookrightarrow \Z \GG^1(\P^n) \simeq \prod_{i=0}^n K(\Z,2i)$$
induces the following map in cohomology
$$h^*(i_{2k}) = \omega^k$$
where $i_{2k}$ is the generator of $H^{2k}(K(\Z,2k),\Z)$ and $\omega$ is the
generator of $H^2(\GG^1(\P^n),\Z)$.
\end{theorem}
\begin{proof} By induction on $n$. \\
{\bf Base:} The case $n=1$ is a result of Lawson and Michelsohn in \cite{LM}. Namely,
they prove that the inclusion $i: \GG^1(\P^1) \hookrightarrow \ZZ^1(\GG^1(\P^1))$ classifies
the total chern class of the universal quotient bundle, i.e. that $i \simeq 1 \times
\omega$. But in this case $\ZZ^1(\G^1(\P^1)) = \Z \G^1(\P^1)$ and $i$ is the 
Hurewicz map $h$. Thus $h \simeq 1 \times \omega$. \\
{\bf Inductive Step:} Notice that $\GG^1(\P^n) = {\P^n}^{\spcheck} \cong
\P^n$ so we will substitute throughout $\GG^1(\P^n)$ with $\P^n$
Suppose that $h^*(i_{2k})=\omega^k$ for $h: \P^n \hookrightarrow \Z \P^n$. The
inclusion of $\P^n \hookrightarrow \P^{n+1}$ is a cofibration and the quotient
$\P^{n+1} / \P^n$ is homeomorphic to the sphere $\S^{2(n+1)}$. Dold and Thom proved
in \cite{DT} that a cofibration sequence induces a quasifibration sequence 
when taking the free abelian group funtor. Hence we have the following 
commutative diagram:
$$
\xymatrix{ 
\P^n \ar@{^{(}->}[r]^j \ar[d]_h  & \P^{n+1} \ar[d]^h \ar[r] & \S^{2(n+1)} \ar[d]^ h \\
\Z \P^n \ar@{^{(}->}[r]   & \Z \P^{n+1} \ar[r]^p & \Z \S^{2(n+1)} 
} 
$$
where $j$ is a cofibration, $p$ is a quasifibration and each $h$ is the 
corresponding Hurewicz map. This diagram is homotopically equivalent to the following
$$
\xymatrix{ 
\P^n \ar@{^{(}->}[r]^j \ar[d]_h  & \P^{n+1} \ar[d]^h \ar[r] & \S^{2(n+1)} \ar[d]^ h \\
\prod_{i=0}^n K(\Z,2i) \ar@{^{(}->}[r]   & \prod_{i=0}^{n+1} K(\Z,2i) \ar[r]^{\pi_{n+1}} &
 K(\Z,2(n+1)) 
} 
$$
The induction hypothesis implies that the vertical arrow on the left satisfies the condition
$h^*(i_{2k})=\omega^k$. Therefore we are only concerned with what happens to
the pullback of $i_{2(n+1)}$ in the middle vertical arrow, but this is determined by the Hurewicz map on the far right
of the diagram. 
\end{proof}
\begin{theorem} \label{nonexistence}
For $p > 1$ there is no continuous biadditive pairing
$$ \rtens: \ZZ^1(\P^n) \times \ZZ^p(\P^m) \rightarrow \ZZ^p(\P^{nm +n +m})$$
such that $\eta \rtens \xi = \eta \otimes \xi$ where $\eta$ and $\xi$ are
linear spaces and $\otimes$ is the map which classifies the tensor product
of bundles via the universal quotient bundle.
\end{theorem}
\begin{proof} Suppose that such a pairing exists. Then it must necessarilly
satisfy the following relation in the degrees 
$$\deg(\eta \rtens \xi)= \deg(\eta)\deg(\xi)$$
This is because it is biadditive and continuous and it maps the degree one
effective cycles into the degree one effective cycles. Thus it induces a 
continuous pairing in the subgroup $\ZZ_0$ of cycles of degree zero:
$$\ZZ_0^1(\P^n) \times \ZZ^p_0(\P^m) \rightarrow \ZZ_0^1(\P^{nm+n+m})$$
Let $\mu:\ZZ_0^1(\P^n) \times \ZZ^p(\P^m) \rightarrow \ZZ_0^1(\P^{nm+n+m})$ 
be the function defined by
$$\mu(\eta,\xi)= \eta \rtens \xi + \eta_0 \rtens \xi + \eta \rtens \xi_0$$
Then the following diagram commutes:
\begin{equation}\label{firstdiagram}\xymatrix{
\GG^1(\P^n) \times \GG^p(\P^m) \ar[r]^{\otimes} \ar[d]_c & \GG^p(\P^{nm+n+m}) \ar[d]^c \\
\ZZ_0^1(\P^n) \times \ZZ^p_0(\P^m) \ar[r]^{\mu}  & \ZZ_0^p(\P^{nm+n+m})
}
\end{equation}
where the vertical maps are the inclusions mapping a linear space $\eta$ into
$\eta -\eta_0$ where $\eta_0$ is a fixed subspace. Lawson and Michelsohn
proved in \cite{LM} that this inclusion classifies the total chern class
map of the universal quotient bundle. Now, notice that we can restrict
the pairing $\mu$ on the first factor to the subspaces $\Z_0\GG^1$ of 
cycles generated by the points of
the grassmannian, that is, to the cycles which are formal
sums of linear hypersurfaces with coefficients adding up to zero. Let $\rho$ be 
the restriction, then we have the following commutative diagram
$$\xymatrix{
\GG^1(\P^n) \times \GG^p(\P^m) \ar[r]^-{\otimes} \ar[d]_{i \times c} & 
\GG^p(\P^{nm+n+m}) \ar[d]^c \\
\Z_0\GG^1(\P^n) \times \ZZ_0^p(\P^m) \ar[r]^-{\rho} \ar[d]_{j \times id} & 
\ZZ_0^p(\P^{nm+n+m}) \ar[d]^{id} \\
\ZZ_0^1(\P^n) \times \ZZ_0^p(\P^m) \ar[r]^-{\mu} & \ZZ_0^p(\P^{nm+n+m})
}
$$
where $i$ is the same map as before, $i(L)=L-L_0$ and $j$ is just the
natural inclusion. The previous theorem gives a description of what
$i$ and $j$ are in terms of the homotopy equivalences with the products
of Eilenberg-Maclane spaces, namely $i$ is the Hurewicz map and $j$ is
the projection onto the first factor. Hence we have the following 
diagram:
\begin{equation}\label{seconddiagram}\xymatrix{
\GG^1(\P^n) \times \GG^p(\P^m) \ar[r]^-{\otimes} \ar[d]_{h \times c} & 
\GG^p(\P^{nm+n+m}) \ar[d]^c \\
\prod_{i=1}^n K(\Z,2i) \times \prod_{i=1}^p K(\Z,2i) \ar[r]^-{\rho} \ar[d]_{\pi_1 \times id} & 
\prod_{i=1}^{nm+n+m} K(\Z,2i) \ar[d]^{id} \\
K(\Z,2) \times \prod_{i=1}^p K(\Z,2i) \ar[r]^-{\mu} & \prod_{i=1}^{nm+n+m} K(\Z,2i) 
}
\end{equation}
We will prove the case $p=2$, the general proof is analogous. The chern class formula for the tensor
product of a line bundle and a $2$-dimensional bundle yields:
$$c_2(L \otimes E) = c_1^2(L) +c_1(E)c_1(L)+c_2(E)$$
The vertical arrows in the diagram  \ref{firstdiagram} induce isomorphisms in
rational cohomology. So the chern class formula implies that in the $4$-th 
cohomology groups $\rho$ should induce the following map:
$$
\rho^*(i_4)=1 \otimes i_4 + i_2 \otimes i_2 + a i_4 \otimes 1 + b i_2^2 \otimes 1
$$
where each $i_k$ is the generator of $H^k(K(\Z,2k);\Q)$ and $a+b=1$. \\
 
We claim that $a=1$ and $b=0$. This claim is the content of
proposition \ref{trick}. The argument to prove
the theorem is then the following:\\
 
The existence of the product $\rtens$ implies the existence
of the function $\mu$ which in turn implies the existence of the restriction
$\rho$. But then, the claim implies that the diagram \ref{seconddiagram} cannot
commute! \\

This is because diagram \ref{seconddiagram} implies that
$$\rho^*(i_4)= (\pi_1 \times id)^*\mu^*(i_4)$$
but there is no element in $H^4(K(\Z,2) \times \prod_{i=1}^2 K(\Z,2i);\Q)$ 
which gets pulled back to
$i_4 \otimes 1$ in $H^4(\prod_{i=1}^n K(\Z,2i) \times \prod_{i=1}^2 K(\Z,2i);\Q)$ 
because $\pi_1$ is the projection into the first factor:
$$\pi_1: \prod_{i=1}^n K(\Z,2i) \rightarrow K(\Z,2)$$
therefore we can only pullback elements of the form $a i_2^p \otimes j$ where
$i_2$ is the generator of 
$$H^*(K(\Z,2);\Q)=\Q[i_2] \subset H^*(\prod_{i=1}^n K(\Z,2i);\Q)=\Q[i_2,\ldots,i_{2n}]$$
(These last equalities being a classical result of Serre).


\end{proof}

Now we prove the claim mentioned in theorem \ref{nonexistence}

\begin{proposition} \label{trick}
Using the notation of theorem \ref{nonexistence} we have the formula
$$
\rho^*(i_4)=1 \otimes i_4 + i_2 \otimes i_2 + i_4 \otimes 1
$$
\end{proposition}
\begin{proof} The chern class formula for the tensor product of bundles and the
commutativity of \ref{seconddiagram} implies that
$$
\rho^*(i_4)=1 \otimes i_4 + i_2 \otimes i_2 + a (i_4 \otimes 1) + b (i_2^2 \otimes 1)
$$
with $a+b=1$. Consider the diagram 
\begin{equation} \label{diagramprime}
\xymatrix{
[\GG^1(\P^n) \times \GG^1(\P^n)] \times \GG^2(\P^m) \ar[r]^-{\otimes_1 \times \otimes_2}
\ar[d]_{\phi \times c} & \GG^2(\P^{mn+n+m}) \times \GG^2(\P^{mn+n+m})  \ar[d]^{c + c} \\
\Z_0 \GG^1(\P^n) \times \ZZ^2_0(\P^m) \ar[r]^-{\rho'} & \ZZ^2_0(\P^{nm+n+m})
}
\end{equation}
where 
\begin{itemize}
\item $\phi:\GG^1(\P^n) \times \GG^1(\P^n)] \rightarrow \Z_0 \GG^1(\P^n)$ is
given by  
$$\phi(L_1,L_2)=(L_1-L_0) + (L_2 - L_0)$$ 
where $L_0$ is a fixed linear space. 
\item $\otimes_1 \times \otimes_2$ is given by
$$(\otimes_1 \times \otimes_2)(L_1,L_2,E)= (L_1 \otimes E, L_2 \otimes E)$$
\item $c + c$ is given by
$$(c + c)(E_1,E_2)= (E_1 - L_0 \otimes E_0) + (E_2 - L_0 \otimes E_0)$$
\item $\rho'= \rho + \tau$ where
$$\tau(\eta,\xi)= L_0 \rtens \xi$$
\end{itemize}

We will verify that diagram \ref{diagramprime} commutes:
\begin{multline*}
(c+c)(\otimes_1 \times \otimes_2)(L_1,L_2,E)=
(c+c)(L_1 \otimes E, L_2 \otimes E) = \\
(L_1 \otimes E - L_0 \otimes E_0) + (L_2 \otimes E - L_0 \otimes E_0) 
\end{multline*}
On the other hand:
\begin{multline*}
\rho'(\phi \times c)(L_1,L_2,E)=
\rho'((L_1-L_0) + (L_2 -L_0),E-E_0) = \\
\rho((L_1-L_0) + (L_2 -L_0),E-E_0)) + \tau((L_1-L_0) + (L_2 -L_0),E-E_0)= \\
\rho((L_1-L_0) + (L_2 -L_0),E-E_0)) + L_0 \rtens (E-E_0) = \\
((L_1-L_0) + (L_2 -L_0)) \rtens (E-E_0) +2( L_0 \rtens(E-E_0)) +
((L_1-L_0) + (L_2 -L_0)) \rtens E_0 =\\
L_1 \rtens E + L_2 \rtens E  -2(L_0 \rtens E_0)
\end{multline*}

Now recall that the space $\ZZ^2_0(\P^{nm+n+m})$ on the lower right corner of diagram
\ref{diagramprime} is homotopically equivalent to $K(\Z,2) \times K(\Z,4)$. We will 
compute the pullback through the whole diagram of the generator $i_4$ of the cohomology
group $H^4(K(\Z,4);\Q)$, considered as a subgroup
$$H^4(K(\Z,4);\Q) \subset H^4(K(\Z,2) \times K(\Z,4);\Q) = \Q i_4 \oplus \Q i_2^2$$

To simplify the notation we will denote by $L_1$, $L_2$ and $E$ the universal
quotient bundles on $\GG^1$, $\GG^1$ and $\GG^2$ correspondingly. Then the
chern class formula for the tensor product and the fundamental result
of \cite{LM} compute the composition $(\otimes_1 \times \otimes_2)^*(c+c)^*$: 
\begin{multline}\label{firstcomposition}
(\otimes_1 \times \otimes_2)^*(c+c)^*(i_4) = (\otimes_1 \times \otimes_2)^*(c_2(E) \otimes c_2(E)) = \\
c_1(L_1)^2 + c_1(L_1)c_1(E) + c_2(E) + 
c_1(L_2)^2 + c_1(L_2)c_1(E) + c_2(E)
\end{multline}

Now, notice that theorem
\ref{hurewicz} implies that in rational cohomology
\begin{equation}\label{pullbackphi}
\phi^*(i_2)= \omega \otimes 1 + 1\otimes \omega \; \; \text{and} \; \;
\phi^*(i_4)= \omega^2 \otimes 1 + 1\otimes \omega^2
\end{equation}

Hence, using the previous equation and the description that we
have for $\rho$ we get
\begin{multline} \label{secondcomposition}
(\phi \times id)^*(\rho')^*(i_4)= (\phi \times id)^*(i_2 \otimes i_2 +
1 \otimes i_4 + b(i_2^2 \otimes 1) + a(i_4 \otimes 1) + 1 \otimes i_4) = \\
c_1(L_1)c_1(E) + c_1(L_2)c_1(E) + c_2(E) + \\
a(c_1(L_1)+c_1(L_2))^2 + b(c_1(L_1)^2 + c_1(L_2)^2) + c_2(E)
\end{multline}

Setting equal the compositions \ref{firstcomposition} and \ref{secondcomposition}
we get that $b=0$ and therefore $a=1$, since there is no term $2c_1(L_1)c_1(L_2)$
in \ref{firstcomposition}.

\end{proof}

                                  %


\begin{bibsection}
\begin{biblist}
\bib{BLLMM}{article}{
author={Boyer, Charles P.},
author={Lawson, H. Blaine, Jr.},
author={Lima-Filho, Paulo},
author={Mann, Benjamin M.},
author={Michelsohn, Marie-Louise},
title={Algebraic cycles and infinite loop spaces},
journal={Invent. Math.},
volume={113},
date={1993},
number={2},
pages={373--388},
issn={0020-9910},
review={\MR{1228130 (95a:55021)}},
}

\bib{DT}{article}{
author={Dold, Albrecht},
author={Thom, Ren{\'e}},
title={Quasifaserungen und unendliche symmetrische Produkte},
language={German},
journal={Ann. of Math. (2)},
volume={67},
date={1958},
pages={239--281},
issn={0003-486X},
review={\MR{0097062 (20 \#3542)}},
}
\bib{LM}{article}{
author={Lawson, H. Blaine, Jr.},
author={Michelsohn, Marie-Louise},
title={Algebraic cycles, Bott periodicity, and the Chern characteristic
map},
conference={
title={The mathematical heritage of Hermann Weyl},
address={Durham, NC},
date={1987},
},
book={
series={Proc. Sympos. Pure Math.},
volume={48},
publisher={Amer. Math. Soc.},
place={Providence, RI},
},
date={1988},
pages={241--263},
review={\MR{974339 (90d:14010)}},
}

\bib{Segaltensor}{article}{
author={Segal, Graeme},
title={Categories and cohomology theories},
journal={Topology},
volume={13},
date={1974},
pages={293--312},
issn={0040-9383},
review={\MR{0353298 (50 \#5782)}},
}
\bib{tot}{article}{
author={Totaro, Burt},
title={The total Chern class is not a map of multiplicative cohomology theories},
journal={Mathematische Zeitschrift},
volume={212},
date={1993},
pages={527--532}, 
}
\end{biblist}
\end{bibsection}
\end{document}